%% file: multsbm_comp_tr.tex
\documentclass[11pt]{article}
\usepackage[letterpaper]{geometry}
\usepackage[parfill]{parskip}
\usepackage{amsmath,amsthm,amssymb,bbm,bm}
\usepackage{mathtools}
\usepackage{cases}
\usepackage{graphicx}
\usepackage{tabularx}
\usepackage{microtype}
\usepackage{enumitem}
\usepackage{dsfont}

\usepackage{authblk}

\usepackage{url}
\usepackage[colorlinks,citecolor=blue,urlcolor=blue,linkcolor=blue,linktocpage=true]{hyperref}
\usepackage{cleveref}
\crefformat{equation}{(#2#1#3)}
\crefrangeformat{equation}{(#3#1#4) to~(#5#2#6)}
\crefname{equation}{}{}
\Crefname{equation}{}{}

\usepackage[numbers]{natbib}
\usepackage{multirow}

\newtheoremstyle{mythmstyle}
  {8 pt} 
  {3 pt} 
  {} 
  {} 
  {\bfseries} 
  {.} 
  {.5em} 
  {} 

\theoremstyle{plain}

\makeatletter
\def\thm@space@setup{%
  \thm@preskip=6pt plus 1pt minus 1pt
  \thm@postskip=\thm@preskip 
}
\makeatother

\newtheorem{theorem}{Theorem}[section]
\newtheorem{lemma}[theorem]{Lemma}

\newtheorem{conjecture}[theorem]{Conjecture}
\newtheorem{remark}{Remark}

\newtheorem*{example*}{Example}
\newtheorem{definition}{Definition}
\newtheorem*{definition*}{Definition}
\newtheorem{assumption}{Assumption}
\newtheorem*{remark*}{Remark}

\crefname{definition}{\textbf{definition}}{definitions}
\Crefname{definition}{Definition}{Definitions}
\crefname{assumption}{\textbf{assumption}}{assumptions}
\Crefname{assumption}{Assumption}{Assumptions}

\newcommand{\bsigma}{\boldsymbol{\sigma}}
\newcommand{\btau}{\boldsymbol{\tau}}
\newcommand{\Mod}[1]{\ (\mathrm{mod}\ #1)}
\begin{document}
\allowdisplaybreaks
\title{Computational and Statistical Thresholds in Multi-layer Stochastic Block Models}

\author[1]{Jing Lei}
\affil[1]{Carnegie Mellon University}

\author[2]{Anru R. Zhang}
\author[2]{Zihan Zhu}
\affil[2]{Duke University}

\maketitle

\begin{abstract}
  We study the problem of community recovery and detection in multi-layer stochastic block models, focusing on the critical network density threshold for consistent community structure inference.  Using a prototypical two-block model, we reveal a computational barrier for such multi-layer stochastic block models that does not exist for its single-layer counterpart: When there are no computational constraints, the density threshold depends linearly on the number of layers. However, when restricted to polynomial-time algorithms, the density threshold scales with the square root of the number of layers, assuming correctness of a low-degree polynomial hardness conjecture.  Our results provide a nearly complete picture of the optimal inference in multiple-layer stochastic block models and partially settle the open question in \cite{lei2022bias} regarding the optimality of the bias-adjusted spectral method.
\end{abstract}

\input{introduction}

\input{prelim}

\input{comp}

\input{info}

\input{disc}

\input{acknowledge}

\appendix

\input{app_a}


\end{document}

%% file: introduction.tex

\section{Introduction}
A network records the pairwise interactions among a group of individuals.  Unlike traditional data that records attributes of individuals, each entry in a network data measurement is indexed by a pair of individuals.  A network data on a group of $n$ individuals is often represented by an $n\times n$ matrix, where the $(i,j)$th entry records the interaction between the individuals $i$ and $j$.  Network data has wide applications in social and nature sciences, and has led to a rich collection of methodological and theoretical developments in statistics and machine learning.  We refer to the books \cite{Goldenberg10,izenman2023network,Kolaczyk09,Newman09} for further readings on general network data analysis.

In recent years, the advancement of measurement technology and data storage capability allows scientists to record multiple networks on the same group of individuals. Such multi-layer networks often reveal more insights about the underlying structure.  For example, \cite{TangLD09} used multi-layer network to analyze the connectivity patterns between research articles using interaction measured in different modes, including title, abstract, and keywords; \cite{DongFVN12} considered cellphone networks defined through tower proximity, blue-tooth proximity, and calls; in \cite{lei2020consistent,lei2022bias,liu2018global}, multiple gene co-expression networks are constructed at different developmental periods.  Other examples include social science \citep{XuH14}, bioinformatics \cite{ZhangC17}, and neural imaging \citep{paul2020random}.  See \cite{KivelaABGMP14} for a general introduction to multi-layer network data.

The emergence of multi-layer network data has inspired many corresponding developments in statistical methods and theory.  As the multi-layer extension of the famous stochastic block model \citep[SBM][]{Holland83}, the multi-layer stochastic block model (MLSBM) is a very popular model for dynamic or multi-layer network data due to its natural ability to model and recover latent communities, a clustering structure that is expected to exist in many real-world networks.  A main challenge in working with multi-layer networks is that the data is now a three-way tensor, where, unlike for matrix-valued single-layer network data, spectral methods are neither easy to compute nor guaranteed to be consistent.  Existing methods deal with the tensor data by either aggregating over the layers \citep{HanXA15,PaulC17,levin2022recovering} or partially vectorizing the tensor to reduce to a matrix form \citep{pensky2019dynamic}.  Likelihood and least-squares-based estimators are also considered but may face computational challenges in the worst-case scenario \citep{MatiasM17,lei2020consistent}.  Recently, \cite{chen2022global} considered a two-stage method that uses likelihood to refine an initial spectral estimate obtained from aggregated layers.

In this work, we aim to provide insights into the understanding of the fundamental limits of community detection and estimation in multi-layer stochastic block models. In single-layer SBMs, it is well known that if the community sizes are balanced, and the edge probabilities are of the same order of magnitude, then consistent recovery of communities is possible if and only if the average degrees diverge as the number of nodes increases, and such recovery can be achieved using variants of spectral methods, which are computationally feasible \citep{AbbeS15,Gao15,LeiR14,ZhangZ16}.  When it comes to multi-layer SBMs, the picture is less clear.  When the number of layers diverges as the number of nodes, the key question is to understand how the signal accumulates over the layers.  In the simplified setting of balanced community sizes and same-order edge probabilities, some recent results suggest that the signal accumulates linearly in the number of layers, and consistent community recovery is possible if and only if the expected total degrees across all layers diverge as numbers of nodes and layers increase \citep{chen2022global,PaulC17}. However, these results require that the layers are all assortative mixing, i.e., nodes in the same community are more likely to connect than nodes from different communities.  There are reasons to believe the assortative mixing assumption may make the problem easier as it provides a simple way to aggregate over the layers, as explained in \cite{lei2020consistent,lei2022bias}.  Without the assortativity assumption, the best known upper bound results suggest that the signal accumulates proportional to the square root of, instead of linearly in, the number of the layers \cite{lei2022bias}.

The main contribution of this work is a characterization of the network density threshold required for consistent community recovery and detection for general multi-layer SBMs.  Interestingly, we find that both the linear and square-root signal accumulation rates are correct, in the sense that the linear rate corresponds to the information-theoretic threshold, whereas the square-root rate corresponds to the computational threshold when restricted to polynomial-time algorithms assuming a hardness conjecture of low-degree polynomials \citep[LDLR,][]{kunisky2019notes}.  Such a computational gap not only reconciles existing results on this topic, it also provides some insights into the source of hardness of inference for multi-layer SBMs.  The proofs of the main results reveal that the computational hardness comes mainly from the unknown \emph{layer identity}. The term ``layer identity'' can be thought of as identifying each layer as assortative mixing or disassortative mixing (the opposite of assortative mixing).  If we are provided with this knowledge, then the signal accumulates linearly over the layers, otherwise, there is a computational cost in finding the best way to aggregate information across the layers.  See \cref{sec:prelim} for a more precise description.

\subsection*{Related work}  There is no gap between computational and information-theoretic thresholds for single-layer SBMs.  Instead, the multi-layer network is more like low-rank tensor data, where similar gaps have been established for certain models, including sparse principal components analysis \citep{berthet2013optimal}, sparse submatrix recovery \citep{ma2015computational}, and tensor SVD \citep{zhang2018tensor}. These computational lower bound results all rely on reducing the statistical inference problem to a conjectured computationally hard problem, usually the planted-clique problem.  A unique challenge posed by the multi-layer SBM model is that the data entries are sparse Bernoulli, which is hard to convert to the dense Bernoulli entries as required in the planted clique model.  Our theoretical development bypasses this difficulty by considering an alternative framework for computational hardness, namely the ``low-degree polynomial'' framework \citep{kunisky2019notes}.  The low-degree polynomial framework grew out of the Sum-of-Squares (SoS) hierarchy \citep{raghavendra2018high}, where the hardness of a problem is gauged by the performance of an increasingly more powerful hierarchy of semidefinite relaxations of polynomial optimization problems.  Roughly speaking, we know that the likelihood ratio is most powerful in distinguishing MLSBMs with different community structures, then it is reasonable to conjecture computational hardness if the best low-degree polynomial approximation to the likelihood ratio cannot distinguish these two models.  See \cite{schramm2022computational} for recent developments of the low-degree polynomial framework in estimation and \cite{luo2023computational} for an application in single-layer graphon estimation.

%% file: prelim.tex

\section{Problem formulation and asymptotic regimes}\label{sec:prelim}
A multi-layer SBM is a probabilistic model for a collection of random graphs observed on a common set of nodes.  For a pair of positive integers $(n,T)$, let $[n]=\{1,2,...,n\}$ be the set of nodes, and $[T]$ the set of layers.  A multi-layer SBM with $K$ common communities generates independent Bernoulli random variables
\begin{equation}
A_{t}(i,j)\stackrel{\rm indep.}{\sim} {\rm Bernoulli}(B_{t}(\sigma_i,\sigma_j))\,,~~1\le i<j\le n\,,~1 \le t\le T\,,\label{eq:MLSBM}
\end{equation}
where
$\bsigma=(\sigma(i):i\in[n])\in[K]^n$ is a community membership vector, with $\sigma(i)\in[K]$ denoting the membership of node $i$, and $B_t\in[0,1]^{K\times K}$ is a symmetric matrix specifying the edge probability between the communities in the $t$th layer.

The observed data are $T$ symmetric binary adjacency matrices $\mathbf A=(A_t:t\in[T])$.  In the theory and practice related to stochastic block models, it is of interest to infer about the membership vector $\bsigma$ from $\mathbf A$.

The hardness of inference problems related to $\bsigma$ depends on $n$, $T$, $\mathbf B=(B_t:t\in[T])$, and $\bsigma$.  In this work we focus on the interplay between three quantities: the number of nodes $n$, the number of layers $T$, and the overall network density.  Here the network density is reflected by the overall magnitude of the entries in $\mathbf B$.

To facilitate our discussion and simplify presentation, we focus on the case $K=2$ and assume both $n$ and $T$ are even.
Let $\mathcal S_n=\{\bsigma\in\{0,1\}^n:\sum_{i\in[n]}\sigma(i)=n/2\}$ be the set of all balanced membership vectors $\bsigma$ on $n$ nodes.  Let $\rho\in(0,2/3)$ be a overall network density parameter.  Define
\begin{equation}
B^{(0)}=\begin{bmatrix}
	\frac{3}{2}\rho & \frac{1}{2}\rho\\
	\frac{1}{2}\rho & \frac{3}{2}\rho
\end{bmatrix}\,,\quad B^{(1)}=\begin{bmatrix}
	\frac{1}{2}\rho & \frac{3}{2}\rho\\
	\frac{3}{2}\rho & \frac{1}{2}\rho
\end{bmatrix}\,.\label{eq:B(1)B(2)}
\end{equation}

\begin{definition}\label{def:model_sequence}
Given the triplet $(n,T,\rho)$, we consider the following two MLSBM models.
\begin{enumerate}
	\item The balanced two-community model $P_{1,n}=P_1(n,T,\rho)$:
\begin{enumerate}
\item $\bsigma\sim {\rm Uniform}(\mathcal S_n)$;
\item $\btau\sim {\rm Uniform}(\mathcal S_T)$;
\item $B_t=B^{(\tau_t)}$ for each $t\in[T]$, with $B^{(0)}$, $B^{(1)}$ given in \eqref{eq:B(1)B(2)};
\item Generate $\mathbf A$ according to \eqref{eq:MLSBM}.
\end{enumerate}
\item The null model $P_{0,n}=P_0(n,T,\rho)$: $$A_t(i,j)\stackrel{\rm indep.}{\sim}{\rm Bernoulli}(\rho)$$ for all $1\le i<j\le n$, $t\in[T]$.
\end{enumerate}	
\end{definition}

We now discuss the rationale behind these choices of mixtures and priors.  First, the randomization on $\bsigma$ is a natural step in establishing minimax lower bounds for estimating $\bsigma$ or detecting the existence of community structure.  Given an arbitrary estimator, we can compute the Bayes risk under this prior in order to provide a minimax lower bound. Such a uniform prior on $\bsigma$ has been used in minimax theory for single layer SBMs \citep{ZhangZ16}.   Second, the randomization of $\btau$ is crucial in distinguishing the information-theoretic threshold and the computational threshold. It turns out that not knowing the layer identity, e.g., the value of $\tau(t)$ for each $t$, can make the inference about $\bsigma$ computationally hard but information-theoretically easy in certain regimes.  These two priors can be changed to the corresponding Bernoulli models (i.e., $\sigma(i)\stackrel{\rm iid}{\sim}{\rm Bernoulli}(1/2)$) with additional bookkeeping. Finally, we include the ``null model'' $P_{0,n}$ in order to consider the detection problem: Whether it is possible to distinguish $P_{1,n}$ from $P_{0,n}$.  The low-degree polynomial framework can be applied to such detection problems hence providing a lower bound for the estimation problem.  Below we will formally introduce the estimation problem and detection problem.

\subsection*{Asymptotic community recovery and detection}  We are interested in the asymptotic behavior of estimators when $n\rightarrow\infty$, $T\rightarrow\infty$, and $\rho\rightarrow 0$.  It is perhaps the easiest to relate $L$ and $\rho$ to $n$ by considering sequences $(T_n,\rho_n)_{n=1}^\infty$, which makes it more natural to write the corresponding mixture model sequence $P_{1,n}=P_{1}(n,T_n,\rho_n)$, and $P_{0,n}=P_0(n,T_n,\rho_n)$ as in \Cref{def:model_sequence}. In the rest of this paper, we will use $P_{1,n}$ to denote the joint distribution of $(\mathbf A,\bsigma,\btau)$. Now we can define the recoverability and distingshuishability of such model sequences.
\begin{definition}[Recoverable and distinguishable model sequences]\label{def:recoverable_detectable}
~
\begin{enumerate}
	\item  For a sequence $(T_n,\rho_n)_{n=1}^\infty$, we say the corresponding MLSBM sequence $P_{1,n}$ defined in \Cref{def:model_sequence} is \emph{recoverable} if there exists a sequence of estimators $\hat\bsigma(\mathbf A)\in\mathcal S_n$ such that
	$$
P_{1,n}(\ell_n(\hat\bsigma,\bsigma)\ge \epsilon)\rightarrow 0
	$$
	for any positive constant $\epsilon$, where $$\ell_n(\hat\bsigma,\bsigma)=n^{-1} \min\left\{d_{\rm Ham}(\hat\bsigma,\bsigma)\,,~d_{\rm Ham}(\hat\bsigma,1-\bsigma)\right\}$$ is the normalized Hamming distance between $\hat\bsigma$ and $\bsigma$ up to label permutation.
	\item For a sequence $(T_n,\rho_n)_{n=1}^\infty$, we say the corresponding MLSBM sequence $(P_{1,n})$ defined in \Cref{def:model_sequence} is \emph{distinguishable} from $P_{0,n}$ if there exists a sequence of  $\hat\psi$: $\hat\psi(\mathbf A)\in\{0,1\}$ such that
	$$
P_{1,n}(\hat\psi(\mathbf A)=0)+P_{0,n}(\hat\psi(\mathbf A)=1)\rightarrow 0\,.
	$$
\end{enumerate}
\end{definition}

Intuitively, detection should be easier than recovery.  We make this formal in the following lemma.
\begin{lemma}[Detection implies recovery]\label{lem:recover=>detect}
Under the MLSBM specified in \Cref{def:model_sequence},	if $P_{1,n}$ is asymptotically recoverable for a sequence $(T_n,\rho_n)$, then it is distinguishable for $(T_n+2,\rho_n)$ with polynomial additional computing time.
\end{lemma}
\Cref{lem:recover=>detect} implies that, asymptotically, the recovery problem is at least as hard as the detection problem.  The change from $T_n$ to $T_n+2$ makes no difference in the asymptotic framework considered in this paper. Therefore, we will provide upper bound results for recovery and lower bound results for detection, and show that there is no gap between recovery and detection for the MLSBM sequence considered.

\subsection*{The range of model parameters}
We will focus on a certain range of $T_n$ and $\rho_n$ that is of main practical interest.  First, it is natural to restrict $\rho_n$ such that $\rho_n=o(1)$ and $\rho_n=\omega(n^{-2})$. 
In fact, the regime $\rho_n=\Omega(\log n/n)$ would make the model $P_{1,n}$ easily recoverable and detectable even when $T=1$ \citep{LeiR14}.  If $\rho_n=o(n^{-2})$, then most of the layers will be completely empty, and it is natural to exclude such a scenario because in practice only non-empty layers are relevant. Second, we will also restrict to $T_n=\omega(1)$, which makes the problem sufficiently different from single layer SBMs.

We summarize the range of $(T_n,\rho_n)$ as below for ease of reference.
\begin{assumption}[Asymptotic regime of $(T_n,\rho_n)$]\label{ass:asymp_regime}
When $n\rightarrow\infty$, we have $T_n\rightarrow\infty$, $\rho_n\rightarrow 0$, $\rho_n^{-1}=o(n^2)$.
\end{assumption}

If we further simplify $(T_n,\rho_n)$ as monomials of $n$, \Cref{ass:asymp_regime} corresponds to 
$$
T_n=n^a\,,~~~ \rho_n=n^{-b}\,,~~~~a>0,~~b\in (0,2)\,.
$$
Then we can summarize our main results in terms of the regime on the parameter space $(a,b)\in (0,\infty)\times (0,2)$.

\begin{theorem}[Simplified main result]
Let $T_n=n^a$, $\rho_n=n^{-b}$ for some $a>0$, $b\in(0,2)$. Then
	\begin{itemize}
		\item without computational constraints, the MLSBM is recoverable if 
		$1+a-b>0$; and is not recoverable by any algorithm if $1+a-b<0$.
		\item assuming the low-degree polynomial conjecture, the MLSBM is recoverable using a polynomial-time algorithm if $1+a/2-b>0$, and not recoverable by any polynomial-time algorithm if $1+a/2-b<0$.
	\end{itemize}
	The above results also hold for detection.
\end{theorem}

%% file: comp.tex

\section{Computational lower and upper bounds}\label{sec:comp}

Now we present the computation thresholds for detection and recovery in the multi-layer stochastic block model. The computational upper bound for recovery (and hence detection) has been studied in \cite{lei2020consistent,lei2022bias}, in which a variant of spectral clustering was proved to consistently recover the membership when $nT_n^{1/2}\rho_n\gtrsim O(\sqrt{\log(n)})$. 

\begin{theorem}[Computational upper bound \cite{lei2020consistent}]
    If $n T_n^{1/2}\rho_n\ge C\sqrt{\log(n)}$ for some absolute constant $C>0$, then MLSBM sequence $(P_{1,n}:n\ge 1)$ specified in \Cref{def:model_sequence} is asymptotically recoverable by a polynomial-time algorithm.
\end{theorem}

Next, we derive a nearly matching computational lower bound for detection in MLSBM, up to a logarithm factor. Our proof is based on the low-degree polynomial argument introduced and developed in \cite{hopkins2018statistical,kunisky2019notes}.  The basic intuition behind this framework is that if a planted hidden structure cannot be detected using low-degree polynomials of the input data, then it cannot be detected by any polynomial time algorithm.  A formal justification of this intuition is stronger than proving ``P$\neq$NP'', and hence instead of a rigorous proof we can only realistically hope to have evidence for such an intuition.  Much evidence in favor of this conjecture has been gathered in the thesis work \cite{hopkins2018statistical}.  On the other hand, many widely implemented polynomial algorithms that achieve the best possible detection thresholds in various related problems, including spectral methods and approximate message passing (see a survey by \cite{feng2022unifying}), can be expressed as low-degree polynomials \citep{gamarnik2020low,kunisky2019notes}.  The low-degree polynomial method can be seen as a ``light version" of the sum-of-squares (SOS) computational lower bounds, which is more straightforward to establish and usually yields the same results for natural average-case hardness problems. Low-degree polynomial computational hardness results have been successfully obtained in a number of problems, such as the planted clique detection \citep{barak2019nearly,hopkins2018statistical}, community detection in stochastic block models \citep{ hopkins2018statistical,hopkins2017efficient}, spiked tensor models \citep{ hopkins2018statistical,hopkins2017power,kunisky2019notes}, spiked Wishart models \citep{bandeira2020computational}, sparse PCA \citep{ding2023subexponential}, spiked Wigner model \citep{kunisky2019notes}, clustering \citep{davis2021clustering,loffler2022computationally,lyu2023optimal}, planted vector recovery \citep{mao2021optimal}, certifying RIP \citep{ding2021average}, independent component analysis \citep{auddy2023large}, random $k$-SAT \citep{bresler2022algorithmic}, and tensor regression \citep{diakonikolas2023statistical,luo2022tensor}. It is also increasingly believed that the low-degree polynomials method can encapsulate the fundamental principles that determine the success or failure of sum-of-squares algorithms \citep{hopkins2018statistical, kunisky2019notes}.

Our argument for the computational lower bound relies on the following low-degree polynomial conjecture, which is a specialized version of the more general statement \citep[Conjecture 2.2.4]{hopkins2018statistical} adapted to our specific context.
\begin{conjecture}[Low-degree polynomial conjecture: Conjecture 2.2.4 in \cite{hopkins2018statistical}]\label{c3.1}
    Let $P_{0,n}$ and $P_{1,n}$ be two sequences of MLSBM defined in \Cref{def:model_sequence} with parameters $(T_n,\rho_n)$ satisfying the asymptotic regime specified in \Cref{ass:asymp_regime}. If every polynomial $\psi$ of degree\footnote{Here the polynomial is viewed as a multivariate polynomial with input vector $\mathbf A\in \mathbb R^{{n\choose 2}\times T_n}$.} at most $D_n=\log^{1.01}(n)$ with $\mathbb{E}_{P_{0,n}}\psi^2 = 1$  and $\mathbb E_{P_{0,n}}\psi=0$ satisfies $\mathbb E_{P_{1,n}}\psi=O(1)$ uniformly, then $P_{1,n}$ is not distinguishable from $P_{0,n}$ by any polynomial-time algorithm.
\end{conjecture}
\begin{remark}
The original version of the conjecture in \cite{hopkins2018statistical} is more general and requires various forms of regularity conditions.  These conditions can be directly verified for our MLSBM sequence $(P_{0,n}, P_{1,n})_{n\ge 1}$ specified in \Cref{def:model_sequence} with $(T_n,\rho_n)$ satisfying \Cref{ass:asymp_regime}.
\end{remark}

\begin{remark}[Intuition of the low-degree polynomial conjecture]
The low-degree polynomial conjecture stems from the classical asymptotic decision theory \citep{le2012asymptotic} and the sum-of-square optimization \citep{hopkins2018statistical}. 
A typical approach to study the statistical distinguishability between $P_{0,n}$ and $P_{1,n}$ is by Le Cam's contiguity \cite{le2012asymptotic}.
Recall that for two sequences of probability distributions, $(P_{0,n},P_{1,n})_{n\in\mathbb N}$, over a common sequence of measurable spaces, we say $P_{1,n}$ is contiguous to $P_{0,n}$, if for any sequence of events $(E_n:n\ge 1)$, $P_{0,n}(E_n) \rightarrow 0$ implies $P_{1,n}(E_n) \rightarrow 0$.
It is well known that contiguity implies statistical impossibility in distinguishing $P_{1,n}$ and $P_{0,n}$ since vanishing type I errors must imply vanishing power. 
Meanwhile, contiguity can be established through the second-moment method, which is essentially the likelihood ratio test:
If $\mathbb{E}_{\mathbf A\sim P_{0,n}}\{[P_{1,n}(\mathbf A)/P_{0,n}(\mathbf A)]^2\}$ remains bounded as $n\rightarrow\infty$, then $P_{1,n}$ is contiguous to $P_{0,n}$.

When there are computational challenges in computing the exact value of the likelihood ratio $L_n(\mathbf A)$, it is natural to study the projection of the likelihood ratio onto the subspace spanned by low-degree polynomials (the exact definition is to be specified in Lemma \ref{l1} below). If the $L^2(P_{0,n})$ norm of this projected likelihood remains bounded when $n\rightarrow\infty$,
it is reasonable to conjecture that no logarithmic-degree polynomial can effectively distinguish $P_{0,n}$ and $P_{1,n}$. 

\end{remark}

Equipped with the low-degree polynomial conjecture, we now present our main result on computational thresholds in MLSBM.
\begin{theorem}[Computational lower bound]\label{t3.1}
    Assuming the low-degree polynomial conjecture \Cref{c3.1}, the MLSBM sequences $(P_{1,n},P_{0,n})$ specified in \Cref{def:model_sequence} with $\rho_n$ satisfying \Cref{ass:asymp_regime} are not distinguishable by any polynomial-time algorithm if
        $n  T_n^{1/2}\rho_n \le (1/2)(\log n)^{-1.4}$ for all $n$ large enough.
\end{theorem}

As detailed in the proof, the exponent of $-1.4$ can be improved to $-1.01$  if $T_n\gg \log^{2.02}n$.  Here the notion of ``polynomial-time'' means polynomial in $n$, because when $\rho_n$ satisfies \Cref{ass:asymp_regime}, the condition $n T_n^{1/2}\rho_n\le (1/2)(\log n)^{1.4}$ implies $T_n=O(n^2)$ and hence the total size of $\mathbf A$ is polynomial in $n$.  Here we describe the sketch of the proof, highlighting some key steps and techniques.  The proof consists of three main steps as summarized in the three lemmas below. The detailed proofs of these lemmas are given in \Cref{pl2}.

The first step is to realize that among all low-degree polynomials of the input data, the projected likelihood ratio has the largest $L^2(P_{0,n})$ norm, and hence is most likely to break contiguity, as required by distinguishability.
    \begin{lemma}\label{l1} (Proposition I.15 in \cite{kunisky2019notes}).
        For any positive integer $n$, given the data $\mathbf A$ and two MLSBM models $P_{0,n}$ and $P_{1,n}$. Let $L_n(\mathbf A) = \frac{P_{1,n}(\mathbf A)}{P_{0,n}(\mathbf A)}$ be the likelihood ratio. 
        Define the norm $\lVert f\rVert\coloneqq \sqrt{\mathbb{E}_{P_{0,n}}f^2(\mathbf A)}$ and define $f^{\leq D}$ as the projection of any function $f$ to the subspace of polynomials of degree at most $D$. For any positive integer $D$, we have
        \[\lVert L_n^{\leq D}(\mathbf A)-1\rVert = \mathop{\sup}_{\substack{\psi:{\rm deg}(\psi)\leq D,\\\mathbb{E}_{P_{0,n}}\psi^2(\mathbf A) = 1,\\\mathbb{E}_{P_{0,n}}\psi(\mathbf A) = 0}}\mathbb{E}_{P_{1,n}}\psi(\mathbf A).\]
    \end{lemma}

The second step is a re-expression of the low-degree likelihood ratio. With \Cref{l1}, we only need to bound $\|L_n^{\le D_n}-1\|$.  In the second step, we use discrete Fourier basis expansion of $L_n^{\le D_n}$ to re-express this quantity in terms of the sizes of certain subsets of the indexing set.  Define the index set $\Lambda_n \coloneqq \{(i_1,i_2,t): 1\leq i_1<i_2\leq n,1\leq t\leq T_n\}$.  For any $\alpha\subseteq\Lambda_n$, and $i\in[n]$, $t\in[T_n]$, define $u_{i,\alpha}=\sum_{(i',j',t')\in\alpha}\mathds{1}(i=i')+\mathds{1}(i=j')$ and $v_{t,\alpha}=\sum_{(i',j',t')\in\alpha}\mathds{1}(t=t')$ be the number of times the node index $i$ and layer index $t$ appear in $\alpha$, respectively. Let $U_\alpha=\{i\in[n]: u_{1,\alpha} \text{ is odd}\}$, and $V_\alpha=\{t\in[T_n]: v_{t,\alpha} \text{ is odd}\}$.  Finally define $\Lambda_{n,a,r,k}=\{\alpha\subseteq\Lambda_n:|\alpha|=a,~|U_\alpha|=2r,~|V_\alpha|=2k\}$.  We have  the following re-expression of $\|L_n^{\le D_n}-1\|$.

\begin{lemma}[Re-expression of LDLR]\label{lem:re-exp}Under the above notation,
            \begin{align*}
 \lVert L_n^{\leq D_n}(\mathbf A)-1\rVert^2
        =\sum_{a = 1}^{D_n}\left(\frac{\rho_n}{4(1-\rho_n)}\right)^{a}\sum_{r = 0}^{a}\sum_{k = 0}^{[a/2]}|\Lambda_{n,a,r,k}|\frac{\binom{n/2}{r}^2\binom{T_n/2}{k}^2}{\binom{n}{2r}^2\binom{T_n}{2k}^2}\,.
        \end{align*}
\end{lemma}

The third and final step is to control the size of the set $\Lambda_{n,a,r,k}$.
\begin{lemma}\label{lem:Lambda}[Controlling $|\Lambda_{n,a,r,k}|$]
When $n$, $T_n$ are large enough, for $a\le D_n=\log^{1.01}n$ we have
$$
|\Lambda_{n,a,r,k}|\le 2^{1+5a/2}a^{4a/3}{n\choose 2r}{T_n\choose 2k}n^{a-r}T_n^{a/2-k}\,.
$$
If $T_n=\omega(D_n^2)$, the term $a^{4a/3}$ in the above bound can be strengthened to $a^a$\,.
\end{lemma}

\begin{proof}[Proof of \Cref{t3.1}]
    According to Conjecture \ref{c3.1}, it suffices to prove that for every polynomial $\psi$ of degree $D_n=\log^{1.01}(n)$ with $\mathbb{E}_{P_{0,n}}\psi^2(\mathbf A) = 1$, the expectation $\mathbb{E}_{P_{1,n}}\psi(\mathbf A)$ is bounded as $n\rightarrow\infty$.  
According to Lemma \ref{l1}, the supremum of $\mathbb{E}_{P_{1,n}}\psi(\mathbf A)$ among all polynomials $\psi$ such that $\deg{\psi} \le D_n$ with zero mean and unit variance under $P_{0,n} $ is equal to $\lVert L_n^{\leq D_n}(\mathbf A)-1\rVert$. Then it suffices to control $\lVert L_n^{\leq D_n}(\mathbf A)-1\rVert$.

Combining \Cref{lem:re-exp} and \Cref{lem:Lambda} we get
\begin{align}
&\lVert L_n^{\leq D_n}(\mathbf A)-1\rVert^2\nonumber\\
        \le&\sum_{a = 1}^{D_n}\frac{\rho_n^a}{(4(1-\rho_n))^a}2^{1+5a/2}a^{4a/3}\sum_{r = 0}^{a}\sum_{k = 0}^{a/2}\frac{\binom{n/2}{r}^2\binom{T_n/2}{k}^2}{\binom{n}{2r}\binom{T_n}{2k}}n^{a-r}T_n^{a/2-k}\nonumber\\
        \le & 8\sum_{a = 1}^{D_n}\frac{\rho_n^a}{(4(1-\rho_n))^a}2^{5a/2}a^{4a/3}n^a T_n^{a/2}
\le  8 \sum_{a = 1}^{D_n}\xi_n^a
\le  \frac{8\xi_n}{1-\xi_n}\,,\label{eq:final_bound_comp_low}
\end{align}
with $\xi_n=2 D_n^{4/3}\rho_n n T_n^{1/2}$.
In the last line of \eqref{eq:final_bound_comp_low}, the first inequality follows from some simple algebra whose detail is given in \Cref{pl2}, the second and third inequalities follow from the sum of geometric sequences and are valid whenever $\rho_n<1/2$ and $\rho_n n T_n^{1/2}<(1/2) D_n^{-4/3}\le (1/2)(\log n)^{-4.04/3}\le (1/2)(\log n)^{-1.4}$, which is guaranteed under the assumption of the theorem for large enough $n$.
\end{proof}


%% file: info.tex

\section{Information-theoretic upper and lower bounds}\label{sec:info}

\subsection{Information-theoretic lower bound}
Our main information-theoretic lower bound result shows that if $nT_n\rho_n$ vanishes then no algorithm can consistently distinguish $P_{1,n}$ and $P_{0,n}$, regardless of the computation power.  Indeed, we will show a stronger result: The community structure is not detectable even when the layer identity $\btau$ is known.

For $\btau\in \mathcal S_{T_n}$, let $P_{1,\btau,n}$ be the corresponding conditional distribution of $(\mathbf A,\bsigma)$ under $P_{1,n}$ given $\btau$.  For $\btau\in\mathcal S_{T_n}$, $\bsigma\in\mathcal S_n$, $P_{\bsigma,\btau}$ denotes the distribution of $\mathbf A$ given $(\bsigma,\btau)$.

Our detection lower bound is established by bounding the $\chi^2$-divergence between $P_{1,\btau,n}$ and $P_{0,n}$.  Recall that for two distributions $P$, $Q$ on the same sample space $\mathcal X$ with probability mass function $q(x)$ and $p(x)$ respectively, their $\chi^2$-divergence is
$$
d_{\chi^2}(Q,P) = \mathbb E_{X\sim P} \left(\frac{q(X)}{p(X)}-1\right)^2 = \sum_{x\in\mathcal X}\frac{q^2(x)}{p(x)}-1\,.
$$
It is well-known that if $d_{\chi^2}(P_{1,n},P_{0,n})\rightarrow 0$ then $P_{1,n}$ and $P_{0,n}$ are indistinguishable \citep[Theorem 2.2]{tsybakov2009introduction}.

\begin{theorem}\label{t3} (Detection lower bound in MLSBM). When $n T_n\rho_n\rightarrow 0$, we have
    $$d_{\chi^2}(P_{1,\btau,n},P_{0,n})=\sum_{\mathbf A}P_{1,\btau,n}^2(\mathbf{A})/P_{0,n}(\mathbf{A})-1= o(1)\,,~~\forall~\btau\in \mathcal S_{T_n}\,.$$
As a result, $$d_{\chi^2}(P_{1,n},P_{0,n})=o(1)$$
and $P_{1,n}$, $P_{0,n}$ are indistinguishable.
\end{theorem}

The proof of \Cref{t3} is given in \Cref{app:proof_info}. The proof is based on a direct expansion of the $\chi^2$-divergence and takes advantage of the symmetry in $P_{1,\btau,n}$ and $P_{0,n}$. Intuitively, the expected number of edges observed for a single node over all layers is approximately $(n-1)T_n\rho_n$.  If $nT_n\rho_n=o(1)$, then for most nodes there will be no observed edges at all.  In this case, it would be impossible to recover the community for the majority of nodes.  In fact, a similar lower bound result for single-layer SBMs has been developed by \cite{ZhangZ16}, using an elegant localization argument that also originates from the symmetry of $P_{1,n}$.  Our result not only extends to multi-layer SBMs but also provides a lower bound for the detection problem, which, as shown in \Cref{lem:recover=>detect}, is at least as hard as the recovery problem.

\subsection{Information-theoretic upper bound}\label{subsec:info-upper}
We consider a variant of the maximum likelihood estimate (MLE) specialized to the MLSBM in \Cref{def:model_sequence}.
Given data $\mathbf A\in \{0,1\}^{{n\choose 2}\times T_n}$, we estimate $\bsigma$ and $\btau$ together by maximizing the blockwise edge count:
\begin{equation}\label{eq:MLE}
(\hat{\bsigma}_{\rm mle},\hat{\btau}_{\rm mle})=\arg\max_{\bsigma\in \mathcal S_n,\btau\in \mathcal S_{T_n}}
\sum_{(i,j,t):2|\sigma(i)+\sigma(j)+\tau(t)}A_t(i,j)\,.
\end{equation}
This estimate extends that of \cite{abbe2015exact}, who studied single-layer exact recovery of the MLE.  

The following result establishes consistency of $\hat\bsigma_{\rm mle}$ when $nT_n\rho_n\rightarrow\infty$.
\begin{theorem}\label{thm:info_upper}
Let $P_{1,n}$ be the sequence of MLSBM defined in \Cref{def:model_sequence} with $(T_n,\rho_n)$ satisfying \Cref{ass:asymp_regime}.
    If $nT_n\rho_n\rightarrow\infty$, then for any $\epsilon>0$
    $$\lim_{n\rightarrow\infty}P_{1,n}(\ell_n(\hat{\bsigma}_{\rm mle},\bsigma)\ge \epsilon)= 0\,,$$
    where $\ell_n(\cdot,\cdot)$ is the Hamming loss function defined in \Cref{def:recoverable_detectable}.  In other words, $P_{1,n}$ is asymptotically recoverable and hence distinguishable from $P_{0,n}$.
\end{theorem}

It is remarkable that our information-theoretic thresholds for recovery and detection have sharp rates: There is no gap in the rates of lower and upper thresholds, and there is no gap between recovery and detection.

\Cref{thm:info_upper} is proved in \Cref{app:proof_info}.  
The basic proof strategy is similar to that of \cite{abbe2015exact}, who studied exact recovery in single-layer SBMs. Our contribution is to extend this argument to the multi-layer case and to provide approximate recovery instead of exact recovery.  The idea is that if the MLE is different from the truth, then there must be a subset of triplets $(i,j,t)$, over which the realized Binomial random variable (i.e., the total edge count) deviates from the expected value by at least a constant fraction.  The probability of this event decays exponentially fast in the Hamming distance between $\hat\bsigma$ and $\bsigma$, which allows a union bound over all possible realizations of $\hat\bsigma$ at any given Hamming distance, and over all possible Hamming distances that are large enough.

%% file: disc.tex

\section{Discussion}\label{sec:disc}
Our computational lower bound result shows that the bias-adjusted spectral clustering applied to the squared adjacency matrices \citep{lei2022bias} achieves nearly the optimal network density threshold among all polynomial-time algorithms assuming the low-degree polynomial conjecture.  Unlike the MLE considered in \Cref{subsec:info-upper}, the bias-adjusted spectral method works for a more general class of MLSBMs, allowing $B_t$ to vary arbitrarily as long as $T_n^{-1}\rho_n^{-2}\sum_t B_t^2\succeq c I$ for some positive constant $c$.  It would be interesting to develop a corresponding method for the information-theoretic upper bound that works for such more general MLSBMs.

Our computational lower bound uses the low-degree polynomial conjecture.  It would be possible to establish computational lower bound results using other frameworks.  The most natural extension, given the success of low-degree polynomial approach used in this manuscript, would be to establish a sum-of-squares lower bound. We conjecture that such a lower bound will be qualitatively the same as our \Cref{t3} due to the close relationship between the LDLR and SoS frameworks, but the technical treatment will be much more involved using SoS.  Another possibility is the reduction technique, which aims to show that solving the community detection/recovery problem in MLSBM is at least as hard as some well-known computationally hard problems, such as the planted clique problem or planted hyperclique problem. There, the sparsity in our MLSBM model could pose some technical challenges.

%% file: acknowledge.tex

\subsection*{Acknowledgement}
The research of ARZ and ZZ was supported in part by NSF Grant CAREER-2203741.  JL's research is partially supported by NSF Grants DMS-2015492, DMS-2310764.

%% file: app_a.tex

\section{Proofs for Section \ref{sec:prelim}}

\begin{proof}[Proof of \Cref{lem:recover=>detect}]
        Suppose $\mathbf A$ is a $(T_n+2)$-layer MLSBM.  Let $\tilde{\mathbf A}$ be obtained by taking the first $T_n$ layers from $\mathbf A$.

By recoverability assumption, there exists an estimate $\hat\bsigma$ such that when applied to $\tilde{\mathbf A}$,
$P_{1,n}(\ell_n(\hat\bsigma,\bsigma)\ge \epsilon)\rightarrow 0$ for any constant $\epsilon>0$.

Now consider the detection rule:
$$
\hat\psi = \left\{\begin{array}{ll}
    1\,, & \text{if } |\frac{4}{n^2}\sum_{i\in\hat\bsigma^{-1}(0),j\in\hat\bsigma^{-1}(1)}A_{T_n+1}(i,j)-\hat\rho_n|\ge 0.3\hat\rho_n\\
 0\,, & \text{otherwise}\,.
\end{array} \right.\,,
$$
where 
$$\hat\rho_n=\frac{\sum_{1\le i<j\le n}A_{T_n+2}(i,j)}{{n\choose 2}}\,.$$

Suppose $\mathbf A\sim P_{1,n}$. Assume for now that $\btau$ is equally split in $[T_n]$. The case of unequal split of $\btau$ in $[T_n]$ will be discussed later.
In this case, with high probability $\hat\bsigma$ differs from $\bsigma$ in at most $\epsilon n$ entries. Let $k=|\hat\bsigma^{-1}(0)\cap \bsigma^{-1}(0)|$ then we have $k\le \epsilon n/2$.

Assume that $\btau(T_n+1)=0$ (The case of $\btau(T_n+1)=1$ can be treated symmetrically).
Then 
$\{A_{T_n+1}(i,j):\hat\bsigma(i)=0,~\hat\bsigma(j)=1\}$
consists of $(n/2-k)^2+k^2$ terms with Bernoulli parameter $\rho_n/2$, and $2k(n/2-k)$ terms with Bernoulli parameter $3\rho_n/2$.  We can pick $\epsilon$ small enough so that the average value of these Bernoulli parameters is less than $0.55 \rho_n$ because when $\epsilon\rightarrow 0$ the average approaches $1/2$.

Now a standard application of Bernstein's inequality to these $n^2/4$ Bernoulli random variables yields
\begin{equation}
\mathbb P\left[\frac{4}{n^2}\sum_{i\in\hat\bsigma^{-1}(0),j\in\hat\bsigma^{-1}(1)}A_{T_n+1}(i,j)\le 0.6\rho_n\right]\ge1- \exp\left[-cn^2\rho_n\right]\,.\label{eq:ave_Tn+1}
\end{equation}

On the other hand, by our assumption, $\btau(T_n+2)=1$. So the node pairs in the layer $A_{T_n+2}$
consist of $(n/2)(n/2-1)$ Bernoulli random variables with parameter $\rho_n/2$, and $n^2/4$ with parameter $3\rho_n/2$.  The total average of these Bernoulli parameters is $\rho_n(1+1/(2n-2))$. Therefore, another round of Bernstein's inequality applied to $\sum_{1\le i<j\le n}A_{T_n+2}(i,j)$ yields
\begin{equation}
\mathbb P\left[|\hat\rho_n-\rho_n|\le 0.1\rho_n \right]\ge 1- 2\exp(-cn^2\rho_n) \label{eq:ave_Tn+2}
\end{equation}
for some absolute constant $c>0$.

Therefore, $\hat\psi=1$ on the intersection of the events in \eqref{eq:ave_Tn+1},  \eqref{eq:ave_Tn+2}, and $\ell_n(\hat\bsigma,\bsigma)\le \epsilon n$, which has probability converging to $1$.

If $\btau$ is not evenly split in $[T_n]$.  We consider a random shuffling of the layers. Under a random shuffling, $\btau$ is evenly split in $[T_n]$ with probability $1/2$.  Therefore, we repeat the random shuffling $M$ times, with $M=M_n$ that diverges to $\infty$ slowly (for example, $M_n=\log (n^2\rho_n)$).  Denote the result $\hat\psi_m$ for the result of the above procedure obtained from the $m$th shuffling.  Take $\hat\psi=\max_m \hat\psi_m$.  Then the probability of $\hat\psi=1$ is reduced by at most $2^{-m}$, the chance that none of the $M$ shuffles produces an even split of $\btau$ in its first $T_n$ entries.

Now we turn to $P_{0,n}$.  For the same reason as above, \eqref{eq:ave_Tn+2} still holds.
But now $A_{T_n+1}$ is also from the null model, and hence we must have
\begin{equation}
\mathbb P\left[\left|\frac{4}{n^2}\sum_{i\in\hat\bsigma^{-1}(0),j\in\hat\bsigma^{-1}(1)}A_{T_n+1}(i,j)-\rho_n\right|\le 0.1\rho_n\right]\ge1- \exp\left[-cn^2\rho_n\right]\,.\label{eq:ave_Tn+1_null}
\end{equation}
It is direct to check that on the intersection of \eqref{eq:ave_Tn+2} and \eqref{eq:ave_Tn+1_null}, $\hat\psi=0$.  Because the probability $\hat\psi=1$ is exponentially small in $n^2\rho_n$, this probability can sustain $M$ rounds of reshuffling using union bound provided that $\log M\ll n^2\rho_n$.
\end{proof} 

\section{Proofs for computational lower bound }\label{pl2}

\begin{proof}[Proof of \Cref{lem:re-exp}]
        The proof starts from decomposing the polynomial $L_n^{\leq D_n}$ by a degree-$D_n$ orthogonal basis under $P_{0,n}$. Recall the definition of the index set $\Lambda_n \coloneqq \{(i_1,i_2,t): 1\leq i_1<i_2\leq n,1\leq t\leq T_n\}$. For any subset $\alpha \subseteq \Lambda_n$ such that $|\alpha|\leq D_n$, we consider the following degree-$|\alpha|$ polynomial:
    \[\chi_\alpha(\mathbf A) = \prod_{(i_1,i_2,t)\in \alpha}\frac{\mathbf A_{t}(i_1,i_2)-\rho_n}{\sqrt{\rho_n(1-\rho_n)}}\,.\]
    Since $\mathbf A_t(i,j)\stackrel{\rm indep.}{\sim}{\rm Bernoulli}(\rho_n)$ under $P_{0,n}$, then it follows that $\mathbb{E}_{P_{0,n}}\chi_{\alpha_1}(\mathbf A)\chi_{\alpha_2}(\mathbf A) = 0$ for any index sets $\alpha_1\neq\alpha_2$, and $\mathbb{E}_{P_{0,n}}\chi_{\alpha}(\mathbf A)^2 = 1$ for any index set $\alpha$. Therefore, $\{\chi_\alpha(\mathbf A):|\alpha|\le D_n\}$ forms a degree-$D_n$ polynomial orthonormal basis, and any mean zero polynomial with a degree no larger than $D_n$  can be represented by a linear combination of this basis:
    \[f(\mathbf A) = \sum_{\alpha\subseteq\Lambda_n:1\leq|\alpha|\leq D_n}c_{\alpha}(f)\chi_\alpha(\mathbf A)\,,\]
    where coefficients $c_\alpha(f) = \mathbb{E}_{P_{0,n}}f(\mathbf A)\chi_\alpha(\mathbf A)$. Recall that $L_n^{\leq D_n}$ is the projection of the likelihood ratio $L_n$ on the subspace of degree-$D_n$ polynomial, then $\lVert L_n^{\leq D_n}(\mathbf A)-1\rVert$ can be rewritten as follows:
        \begin{align}
            \lVert L_n^{\leq D_n}(\mathbf A)-1\rVert^2 =& 
            \sum_{1\leq|\alpha|\leq D_n}\left[\mathbb{E}_{P_{0,n}}(L_n^{\leq D_n}(\mathbf A)-1)\cdot\chi_{\alpha}(\mathbf A)\right]^2\nonumber\\
            =&
            \sum_{1\leq|\alpha|\leq D_n}\left[\mathbb{E}_{P_{0,n}}(L_n(\mathbf A)-1)\cdot\chi_{\alpha}(\mathbf A)\right]^2\nonumber\\
            =&
            \sum_{1\leq|\alpha|\leq D_n}(\mathbb{E}_{P_{0,n}}L_n(\mathbf A)\cdot\chi_\alpha(\mathbf A))^2\nonumber\\
            =&
            \sum_{1\leq|\alpha|\leq D_n}(\mathbb{E}_{P_{1,n}}\chi_\alpha(\mathbf A))^2\,,\label{eq6}
        \end{align}
    where the first equality uses the property of orthonormal basis expansion, the second uses the construction of $L_n^{\le D_n}$, and the third uses $\mathbb E_{P_{0,n}}\chi_\alpha(\mathbf A)=0$ for all non-empty $\alpha$, and the last one uses the property of likelihood ratio.
    
    Next, we control $\sum_{1\leq|\alpha|\leq D_n}(\mathbb{E}_{P_{1,n}}\chi_\alpha(\mathbf A))^2$. To begin with, We decouple the expectation $\mathbb{E}_{P_{1,n}}\chi_\alpha(\mathbf A)$ as $\mathbb{E}_{\bsigma,\btau}\mathbb{E}\{\chi_{\alpha}(\mathbf A)|\bsigma,\btau\}$. Given the membership $\bsigma,\btau$, $A_t(i_1,i_2)$ are independent for any $(i_1,i_2,t)\in\Lambda_n$, namely,
    \[\mathbb{E}\{\chi_{\alpha}(\mathbf A)|\bsigma,\btau\} = \prod_{(i_1,i_2,t)\in \alpha}\mathbb{E}\left\{\frac{A_t(i_1,i_2)-\rho_n}{\sqrt{\rho_n(1-\rho_n)}}\bigg|\bsigma,\btau\right\}.\]
    Also, we remark that the following equation holds given the membership $\bsigma,\btau$:
    \begin{equation}\nonumber
        \mathbb{E}\left\{\frac{A_t(i_1,i_2)-\rho_n}{\sqrt{\rho_n(1-\rho_n)}}\bigg|\bsigma,\btau\right\} = \left\{
        \begin{array}{l}
            \frac{\rho_n/2}{\sqrt{\rho_n(1-\rho_n)}},~~\text{ if } 2\mid \bsigma(i_1)+\bsigma(i_2)+\btau(t); \\
            \frac{-\rho_n/2}{\sqrt{\rho_n(1-\rho_n)}},~~\text{ if } 2\nmid \bsigma(i_1)+\bsigma(i_2)+\btau(t)\\
        \end{array}
        \right.
    \end{equation}
    Combining the above equations, we calculate $\mathbb{E}_{P_{1,n}}\{\chi_\alpha(\mathbf A)|\bsigma,\btau\}$ as follows,
    \begin{equation}\nonumber
        \begin{aligned}
            \mathbb{E}_{P_{1,n}}\left\{\chi_\alpha(\mathbf A)|\bsigma,\btau\right\} &= \mathbb{E}\left\{\prod_{(i_1,i_2,t)\in\alpha}\chi_\alpha(\mathbf A_t(i_1,i_2))\big|\bsigma,\btau\right\}\\
            &= \prod_{(i_1,i_2,t)\in\alpha}\left\{\left(\frac{\rho_n/2}{\sqrt{\rho_n(1-\rho_n)}}\cdot (-1)^{\bsigma(i_1)+\bsigma(i_2)+\btau(t)}\right)\right\}\\
            &= \left(\frac{\rho_n/2}{\sqrt{\rho_n(1-\rho_n)}}\right)^{|\alpha|}\cdot (-1)^{\sum_{(i_1,i_2,t)\in\alpha}\left\{\bsigma(i_1)+\bsigma(i_2)+\btau(t)\right\}}.
        \end{aligned}
    \end{equation}
    Given the index set $\alpha$, for each $i\in[n]$ and $t\in[T_n]$, define $u_{i,\alpha} \coloneqq \sum_{(i_1^\prime,i_2^\prime,t^\prime)\in\alpha}\mathbf {1}_{i\in\{i_1^\prime,i_2^\prime\}}$ and $v_{t,\alpha} \coloneqq \sum_{(i_1^\prime,i_2^\prime,t^\prime)\in\alpha}\mathbf {1}_{t = t^\prime}$. Namely, $u_{i,\alpha}$ is the number of times that $i$ appears in $\alpha$ and $v_{t,\alpha}$ is the number times that $t$ appears in $\alpha$. Using this notation, we have $\sum_{(i_1,i_2,t)\in\alpha}\left\{\bsigma(i_1)+\bsigma(i_2)+\btau(t)\right\} \equiv \sum_{i\in[n],2\nmid u_{i,\alpha}}\bsigma(i)+\sum_{t\in[T_n],2\nmid v_{t,\alpha}}\btau(t) \Mod{2}$. Then we have
    \begin{equation}\label{eq7}
        \mathbb{E}_{P_{1,n}}\left\{\chi_\alpha(\mathbf A)|\bsigma,\btau\right\}= \left(\frac{\rho_n/2}{\sqrt{\rho_n(1-\rho_n)}}\right)^{|\alpha|}\cdot (-1)^{\sum_{i\in[n],2\nmid u_{i,\alpha}}\bsigma(i)+\sum_{t\in[T_n],2\nmid v_{t,\alpha}}\btau(t)}.
    \end{equation}
    Define $\kappa \coloneqq (\rho_n/2)/\sqrt{\rho_n(1-\rho_n)}$ and recall that $\mathbb{E}_{P_{1,n}}\chi_\alpha(\mathbf A) = \mathbb{E}_{\bsigma,\btau}\mathbb{E}\{\chi_{\alpha}(\mathbf A)|\bsigma,\btau\}$, we calculate $\mathbb{E}_{P_{1,n}}\chi_\alpha(\mathbf A)$ as follows,
    \begin{equation}\label{eq8}
        \begin{aligned}
            \mathbb{E}_{P_{1,n}}\chi_\alpha(\mathbf A) &= \frac{1}{\binom{n}{n/2}\binom{T_n}{T_n/2}}\sum_{\bsigma\in\mathcal S_n,\btau\in\mathcal S_{T_n}}\mathbb{E}\{\chi_{\alpha}(\mathbf A)|\bsigma,\btau\}\\
            &= \frac{\kappa^{|\alpha|}}{\binom{n}{n/2}\binom{T_n}{T_n/2}}\sum_{\bsigma\in\mathcal S_n,\btau\in\mathcal S_{T_n}}(-1)^{\sum_{i\in[n],2\nmid u_{i,\alpha}}\bsigma(i)+\sum_{t\in[T_n],2\nmid v_{t,\alpha}}\btau(t)}\\
            &= \frac{\kappa^{|\alpha|}}{\binom{n}{n/2}\binom{T_n}{T_n/2}}\sum_{\bsigma\in\mathcal S_n}(-1)^{\sum_{i\in[n],2\nmid u_{i,\alpha}}\bsigma(i)}\cdot\sum_{\btau\in\mathcal S_{T_n}}(-1)^{\sum_{t\in[T_n],2\nmid v_{t,\alpha}}\btau(t)}.
        \end{aligned}
    \end{equation}
    For a given $\alpha$, recall the definition $U_\alpha\coloneqq\{i\in[n]:2\nmid u_{i,\alpha}\}$ and $V_\alpha\coloneqq\{t\in[T_n]:2\nmid v_{t,\alpha}\}$. We can rewrite \eqref{eq8} as follows, 
    \[\mathbb{E}_{P_{1,n}}\chi_\alpha(\mathbf A)=\frac{\kappa^{|\alpha|}}{\binom{n}{n/2}\binom{T_n}{T_n/2}}\sum_{\bsigma\in\mathcal S_n}(-1)^{\sum_{i\in U_\alpha}\bsigma(i)}\cdot\sum_{\btau\in\mathcal S_{T_n}}(-1)^{\sum_{t\in V_\alpha}\btau(t)}.\]

    Let $V_{\alpha,1}=\{t\in V_{\alpha}: \btau(t)=1\}$, and $|V_{\alpha,1}|=r\in[0,|V_\alpha|]$.
    Then $(-1)^{\sum_{t\in V_\alpha}\btau(t)}=(-1)^r$.  For a fixed $r\in[0,|V_\alpha|]$, the number of $\btau\in\mathcal S_{T_n}$ with $|V_{\alpha,1}|=r$ equals ${|V_\alpha|\choose r}{T_--|V_\alpha|\choose T_n/2-r}$. Thus we have 
    \begin{equation}\label{eq10}
            \sum_{\btau\in\mathcal S_{T_n}}(-1)^{\sum_{t\in V_\alpha}\btau(t)} =
            \sum_{r = 0}^{|V_\alpha|}(-1)^r\binom{|V_\alpha|}{r}\binom{T_n-|V_\alpha|}{T_n/2-r}\,.
    \end{equation}
    Similarly, we decompose $\sum_{\bsigma\in\mathcal n}(-1)^{\sum_{i\in U_\alpha}\bsigma(i)}$ by the following equation:
    \begin{equation}\label{eq11}
        \begin{aligned}
            \sum_{\bsigma\in\mathcal S_n}(-1)^{\sum_{i\in U_\alpha}\bsigma(i)} = \sum_{s = 0}^{|U_\alpha|}(-1)^s\binom{|U_\alpha|}{s}\binom{n-|U_\alpha|}{n/2-s}.
        \end{aligned}
    \end{equation}
    Combining \eqref{eq10} and \eqref{eq11}, we calculate $ \mathbb{E}_{P_{1,n}}\chi_\alpha(\mathbf A)$ as follows:
        \begin{align}
            &\mathbb{E}_{P_{1,n}}\chi_\alpha(\mathbf A)\nonumber\\
            =&\frac{\kappa^{|\alpha|}}{\binom{n}{n/2}\binom{T_n}{T_n/2}}\sum_{s = 0}^{|U_\alpha|}(-1)^s\binom{|U_\alpha|}{s}\binom{n-|U_\alpha|}{n/2-s}\cdot\sum_{r = 0}^{|V_\alpha|}(-1)^r\binom{|V_\alpha|}{r}\binom{T_n-|V_\alpha|}{T_n/2-r}\nonumber\\
            \stackrel{\displaystyle (a)}{=} & \frac{\kappa^{|\alpha|}}{\binom{n}{|U_\alpha|}\binom{T_n}{|V_\alpha|}}\sum_{s = 0}^{|U_\alpha|}(-1)^s\binom{n/2}{s}\binom{n/2}{|U_\alpha|-s}\cdot\sum_{r = 0}^{|V_\alpha|}(-1)^r\binom{T_n/2}{r}\binom{T_n/2}{|V_\alpha|-r}\,,\label{eq12}
        \end{align}
    where $\displaystyle (a)$ is guaranteed by the fact
    \[\frac{\binom{|U_\alpha|}{s}\binom{n-|U_\alpha|}{n/2-s}}{\binom{n}{n/2}} = \frac{\frac{(n/2)!}{s!(n/2-s)!}\cdot\frac{(n/2)!}{(|U_\alpha|-s)!(n/2+s-|U_\alpha|)!}}{\frac{n!}{|U_\alpha|!(n-|U_\alpha|)!}} = \frac{\binom{n/2}{s}\binom{n/2}{|U_\alpha|-s}}{\binom{n}{|U_\alpha|}}.\]
    Now we propose the following useful lemma to calculate \eqref{eq12}, we refer the reader to Appendix \ref{pl2} for a complete proof.
 
    According to Lemma \ref{l2}, $\mathbb{E}_{P_{1,n}}\chi_\alpha(\mathbf A) \neq 0$ only if $2\mid |U_\alpha|$ and $2\mid |V_\alpha|$,  and in this case we have
    \begin{equation}\label{9}
        \mathbb{E}_{P_{1,n}}\chi_\alpha(\mathbf A) = \frac{\kappa^{|\alpha|}(-1)^{|U_\alpha|/2+|V_\alpha|/2}\binom{n/2}{|U_\alpha|/2}\binom{T_n/2}{|V_\alpha|/2}}{\binom{n}{|U_\alpha|}\binom{T_n}{|V_\alpha|}}\,.
    \end{equation}
    Plugging \eqref{9} into \eqref{eq12} and further into \eqref{eq6}, we obtain the following expression for $\lVert L_n^{\leq D_n}(\mathbf A)-1\rVert^2$:
        \begin{align}
            \lVert L_n^{\leq D_n}(\mathbf A)-1\rVert^2 &= \sum_{1\leq|\alpha|\leq D_n}(\mathbb{E}_{P_{1,n}}\chi_\alpha(\mathbf A))^2 = 
            \sum_{a = 1}^{D_n}\sum_{\alpha\subseteq\Lambda_n,|\alpha| = a}(\mathbb{E}_{P_{1,n}}\chi_\alpha(\mathbf A))^2\nonumber\\
            &= \sum_{a = 1}^{D_n}\kappa^{2a}\sum_{\substack{|\alpha| = a\\2\mid |U_\alpha|, 2\mid |V_\alpha|}}\frac{\binom{n/2}{|U_\alpha|/2}^2\binom{T_n/2}{|V_\alpha|/2}^2}{\binom{n}{|U_\alpha|}^2\binom{T_n}{|V_\alpha|}^2}\,.\label{10}
        \end{align}

    Recall the definition $\Lambda_{n,a,r,k} \coloneqq \{\alpha\subseteq\Lambda_n: |\alpha| = a,|U_\alpha| = 2r,|V_{\alpha}| = 2k\}$.  Using this notation, \eqref{10} can be rewritten as follows:
    \begin{equation*}
        \lVert L_n^{\leq D_n}(\mathbf A)-1\rVert^2
        = \sum_{a = 1}^{D_n}\kappa^{2a}\sum_{r = 0}^{a}\sum_{k = 0}^{[a/2]}|\Lambda_{n,a,r,k}|\frac{\binom{n/2}{r}^2\binom{T_n/2}{k}^2}{\binom{n}{2r}^2\binom{T_n}{2k}^2}\,. \qedhere
    \end{equation*}
\end{proof}

\begin{proof}[Proof of \Cref{lem:Lambda}]
In this proof, we write $T$ for $T_n$ to simplify the notation.

To begin with, define $\Omega_{n,a,r,k} \coloneqq \{\alpha\subseteq[n]\times[n]\times[T_n]: |\alpha| = a,|U_\alpha| = 2r,|V_{\alpha}| = 2k\}$, where  the definition of $U_\alpha$ and $V_\alpha$ is the same as in the definition of $\Lambda_{n,a,r,k}$. We remark that the only difference between $\Omega_{n,a,r,k}$ and $\Lambda_{n,a,r,k}$ is that for any $\alpha\in\Omega_{n,a,r,k}$ and any tuple $(i_1,i_2,t)\in\alpha$, we do not require $i_1<i_2$. Therefore,
    \[|\Lambda_{n,a,r,k}|\leq |\Omega_{n,a,r,k}|.\]

In the remaining of this proof, we bound $|\Omega_{n,a,r,k}|$.

For an arbitrary $\alpha\subseteq[n]\times [n]\times [T]$, define
$x_\alpha(i)=\sum_{(i_1,i_2,t)\in\alpha}\left(\mathds{1}(i_1=i)+\mathds{1}(i_2=i)\right)$, i.e., the total number of times $i$ appears in $\alpha$.  Similarly, define $y_\alpha(t)=\sum_{(i_1,i_2,t_1)\in\alpha}\mathds{1}(t=t_1)$.

Recall that a multi-set is a set that keeps track of the multiplicities of each element. For example, a multiple set $\{1,1,2\}$ is the same as $\{1,2,1\}$ but different from $\{1,2\}$ and $\{1,2,2\}$.

Let $\mathcal I_{a,r}$ be the collection of all multi-sets $\beta\subseteq [n]\times [n]$ such that (i) there are exactly $2r$ $x_\beta(i)$'s being odd,
and (ii) $|\beta|=a$,
where $x_\beta(i)$ is defined similarly as $x_\alpha(i)$.

Let $\mathcal T_{a,k}$ be the collection  of all multi-sets $\gamma\subseteq [T]$ such that
(i) there are exactly $2k$ $y_\gamma(t)$'s being odd,
  and (ii) $|\gamma|=a$.

For each $\alpha\in\Omega_{a,r,k}$. Write $\alpha=\{(i_1,j_1,t_1),...,(i_a,j_a,t_a)\}$. Define 
$\beta=\{(i_1,j_1),...,(i_a,j_a)\}$ and $\gamma=\{t_1,...,t_n\}$. Then $\beta\in \mathcal I_{a,r}$ and $\gamma\in\mathcal T_{a,k}$.
This defines a mapping $h:\Omega_{a,r,k}\mapsto \mathcal I_{a,r}\times\mathcal T_{a,k}$.
Given $\beta\in\mathcal I_{a,r}$ and $\gamma\in\mathcal T_{a,k}$, there are at most $a!$ ways to pair them, and the claimed result follows.  So $|h^{-1}(\beta,\gamma)|\le a!$ for each $(\beta,\gamma)\in \mathcal I_{a,r}\times\mathcal T_{a,r}$. Thus we arrive at the following
\begin{equation}\label{eq:Omega < I times T}
|\Omega_{n,a,r,k}|\le |\mathcal I_{a,r}|\times|\mathcal T_{a,k}|\times a!\,.\end{equation}

\textbf{Bounding $|\mathcal I_{a,r}|$.}
Now let $\mathcal J_{a,r}$ be the collection of multi-sets consisting of elements in $[n]$ such that for each $\delta\in \mathcal J_{a,r}$, (i) $|\delta|=2a$, and (ii)
 there are exactly $2r$ $x_\delta(i)$'s being odd, where $x_\delta(i)=\sum_{j\in \delta}\mathds{1}(i=j)$.

Each $\beta=\{(i_1,j_1),...,(i_a,j_a)\}\in \mathcal I_{a,r}$ maps to a $\delta=\{i_1,j_1,...,i_a,j_a\}\in \mathcal J_{a,r}$.  For $\delta\in\mathcal J_{a,r}$, the inverse mapping is pairing up the elements of $\delta$, and there are at most ${2a\choose a}$ possible ways. So we have
\begin{equation}\label{eq:I < J}
|\mathcal I_{a,r}|\le |\mathcal J_{a,r}|{2a\choose a}\,.    
\end{equation}

Now we decompose $\mathcal J_{a,r}=\bigsqcup_{u=0}^{a-r}\mathcal J_{a,r,u}$, where $\mathcal J_{a,r,u}$ consists of elements in $\mathcal J_{a,r}$ in which the odd indices appear a total of $2r+2u$ times. 
We construct an element $\delta$ in $\mathcal J_{a,r,u}$ in three steps. In order to specify an element $\delta\in\mathcal J_{a,r}$, it is equivalent to specify $x_\delta(i)$ for each $i\in[n]$. First, we pick $2r$ indices in $[n]$ for which $x_\delta(i)$ is odd.  There are 
${n\choose 2r}$ ways for this step.  Second, we specify $x_\delta(i)$ for these $2r$ indices.  
This corresponds to solving the equation $z_1+...+z_{2r}=u$ over non-negative integers $z_1,...z_{
2r}$.  The number of
ways of doing this is ${u+2r-1\choose u}$.  Third, we need to specify $x_\delta(i)$ for the 
remaining $n-2r$ indices.  This corresponds to solving $z_1+...+z_{n-2r}=a-r-u$.  There are ${n+a-
3r-u-1\choose a-r-u}$ different ways.  So we arrive at
\begin{align}
|\mathcal I_{a,r}|\le & {2a\choose a} |\mathcal J_{a,r}|\nonumber\\
 \le &{2a\choose a}{n\choose 2r}\sum_{u=0}^{a-r}{u+2r-1\choose u}{n+a-3r-u-1\choose a-r-u}\label{eq:J_a,r,u-decomp}\\
 \le & 2 {2a\choose a}(n+a-3r-1)^{a-r}\nonumber\\
 \le & (1+o(1))2{2a\choose a}(n+a-1)^{a-r}\,,\label{eq:I_a,r}
\end{align}
where the third inequality follows by the fact that when $a/n\le D_n/n$ is sufficiently small, the summands in the RHS are upper bounded by a geometric sequence whose decay rate is at least $1/2$ so that the sum is bounded by twice of the first term.



Now we turn to $|\mathcal T_{a,k}|$.
Let $a_T=\min(a, T)$. For the same decomposition and three-step construction for $\mathcal J_{a,r,u}$ we have
\begin{align}
|\mathcal T_{a,k}|\le & \sum_{v=0}^{a_T/2-k}  {T\choose 2k}{2k+v-1\choose v}{T+a_T/2-3k-v-1\choose a_T/2-k-v}\,. \label{eq:|T_a,k|-initial}
\end{align}
If $T\gg a^2$, then $a_T=a$ and the same argument as in \eqref{eq:I_a,r} leads to
\begin{equation}\label{eq:|T_a,k|-large-T}
|\mathcal T_{a,k}|\le 2{T\choose 2k}(T+a/2-3k-1)^{a/2-k}\,.
\end{equation}

In general, \eqref{eq:|T_a,k|-initial} can be bounded by
\begin{align}
    |\mathcal T_{a,k}|\le & {T\choose 2k}(a_T/2-k){a_T/2+k-1\choose a_T/2-k}{T+a_T/2-3k-1\choose a_T/2-k}\nonumber\\
    \le & {T\choose 2k}(a_T/2-k)(a_T/2+k-1)^{(a_T/2-k)\wedge (2k-1)}{T+a_T/2-3k-1\choose a_T/2-k}\nonumber\\
    \le & {T\choose 2k}(a_T/2-k)(a_T/2+k-1)^{a_T/3-1}{T+a_T/2-3k-1\choose a_T/2-k}\nonumber\\
    \le & {T\choose 2k}a_T^{a_T/3}(T+a_T/2-1)^{a_T/2-k}\label{eq:|T_a,k|-general}\,.
\end{align}

Plugging \eqref{eq:I_a,r} and \eqref{eq:|T_a,k|-general} into \eqref{eq:Omega < I times T}, we have
arrive at the bound
\begin{align*}
    |\Omega_{n,a,r,k}| \le &(1+o(1))2\frac{(2a)!}{a!}a_T^{a_T/3}{n\choose 2r}{T\choose 2k}(n+a-1)^{a-r}(T+a_T/2-1) ^{a_T/2-k}\\
    \le &2^{1+a}a^{4a/3}{n\choose 2r}{T\choose 2k}(n+a-1)^{a-r}(T+a_T/2-1) ^{a_T/2-k}\\
    \le &2^{1+5a/2}a^{4a/3}{n\choose 2r}{T\choose 2k}n^{a-r}T^{a/2-k}\,,
\end{align*}
where the last inequality uses $n+a-1\le 2n$ and $T+a_T/2-1\le 2T$.

When $T\gg a^2$, the bound strengthens to
\begin{align*}
&|\Omega_{n,a,r,k}|\le 2^{1+5a/2}a^a{n\choose 2r}{T\choose 2k}n^{a-r}T^{a/2-k}\,.\qedhere
\end{align*}
\end{proof}

   \begin{lemma}\label{l2} For any positive integer $m$ and $0\leq k\leq m$, the following equation holds:
        \[\sum_{i = 0}^k(-1)^i\binom{m}{i}\binom{m}{k-i} = \left\{
        \begin{aligned}
            &0,~~&\text{if}~~2\nmid k; \\
            &(-1)^{k/2}\binom{m}{k/2},~~&\text{if}~~2\mid k.\\  
        \end{aligned}
        \right.\]
    \end{lemma}
\begin{proof}[Proof of Lemma \ref{l2}]
    For any positive integer $m$ and $0\leq k\leq m$, we consider the following polynomial:
\[g(x) = (1-x)^m(1+x)^m.\]
We compute the coefficient of the $x^k$ term in $g(x)$, where we leverage the following form,
\[g(x) = \left(\sum_{i = 0}^m\binom{m}{i}(-1)^ix^i\right)\left(\sum_{i = 0}^m\binom{m}{i}x^i\right).\]
Therefore, the coefficient for $x^k$ is $\sum_{i = 0}^k(-1)^i\binom{m}{i}\binom{m}{k-i}$. 

Besides, we remark that $g(x) = (1-x^2)^m$, which further implies that
\[g(x) = \sum_{i = 0}^m\binom{m}{i}(-1)^ix^{2i}.\]
Thus, the coefficient of $x^k$ is zero if $2\nmid k$, and the coefficient is $(-1)^{k/2}\binom{m}{k/2}$ if $2\mid k$.
\end{proof}

\subsection*{Details for \eqref{eq:final_bound_comp_low}}
\begin{proof}
Combining \Cref{lem:re-exp} and \Cref{lem:Lambda} we get
\begin{align}
&\lVert L_n^{\leq D_n}(\mathbf A)-1\rVert^2\nonumber\\
        \le&\sum_{a = 1}^{D_n}\frac{\rho_n^a}{(4(1-\rho_n))^a}2^{1+5a/2}a^{4a/3}\sum_{r = 0}^{a}\sum_{k = 0}^{a/2}\frac{\binom{n/2}{r}^2\binom{T_n/2}{k}^2}{\binom{n}{2r}\binom{T_n}{2k}}n^{a-r}T_n^{a/2-k}\nonumber\\
        =&\sum_{a = 1}^{D_n}\frac{\rho_n^a}{(4(1-\rho_n))^a}2^{1+5a/2}a^{4a/3}\left[\sum_{r = 0}^{a}\frac{\binom{n/2}{r}^2}{\binom{n}{2r}}n^{a-r}\right]\left[\sum_{k = 0}^{a/2}\frac{\binom{T_n/2}{k}^2}{\binom{T_n}{2k}}T_n^{a/2-k}\right]\,.\label{eq:final_detail_comp_lower}
\end{align}
It suffices to provide upper bounds for the two inner sums.

For $r\le a\le D_n = o(n)$, the terms in the sequence
$$
\frac{{n/2\choose r}^2}{{n\choose 2r}}n^{a-r}\,,\quad r=0,...,a
$$
are bounded by a geometric sequence with a decay rate of at least $1/2$ by Stirling's formula. Therefore their sum is bounded by twice the first term and we get
\begin{align}
  \sum_{r = 0}^{a}\frac{\binom{n/2}{r}^2}{\binom{n}{2r}}n^{a-r}\le 2 n^a\,.\label{eq:detail_sum_r}
\end{align}

By Stirling's formula, when $1\le k\le T/2-1$, we have
$$
\frac{{T/2\choose k}^2}{{T\choose 2k}}\le c T^{1/2}
$$
for some universal constant $c$. Therefore
\begin{align}
&\sum_{k = 0}^{a/2}\frac{\binom{T_n/2}{k}^2}{\binom{T_n}{2k}}T^{a/2-k}
\le T_n^{a/2}+1+c T_n^{1/2}\sum_{k = 1}^{(a/2)\wedge (T_n/2-1)}T_n^{a/2-k}\nonumber\\
\le& T_n^{a/2}+1+2c T_n^{1/2}T_n^{a/2-1}
\le 2 T_n^{a/2}\,.\label{eq:detail_sum_k}
\end{align}
Plugging \eqref{eq:detail_sum_r} and \eqref{eq:detail_sum_k} into \eqref{eq:final_detail_comp_lower} leads to the desired result.
\end{proof}

\section{Proofs for information-theoretic bounds}\label{app:proof_info}
\begin{proof}[Proof of \Cref{t3}]
First we expand the $\chi^2$ divergence between $P_{1,\btau,n}$ and $P_{0,n}$.
    \begin{equation}\label{eq21}
        \begin{aligned}
            &\sum_{\mathbf{A}}\frac{P_{1,\btau,n}^2(\mathbf{A})}{P_{0,n}(\mathbf{A})}\\
            = & \frac{1}{\binom{n}{n/2}^2}\sum_{\mathbf{A}}\frac{\sum_{\bsigma_1,\bsigma_2}P_{\bsigma_1,\btau}(\mathbf{A})P_{\bsigma_2,\btau}(\mathbf{A})}{P_{0,n}(\mathbf{A})} = \frac{1}{\binom{n}{n/2}^2}\sum_{\bsigma_1}\sum_{\bsigma_2}\sum_{\mathbf{A}}\frac{P_{\bsigma_1,\btau}(\mathbf{A})P_{\bsigma_2,\btau}(\mathbf{A})}{P_{0,n}(\mathcal{\mathbf{A}})}\\
            = & \frac{1}{\binom{n}{n/2}^2}\sum_{\bsigma_1,\bsigma_2}\prod_{1\le i<j\le n,t\in[T_n]}\left(\frac{P_{\bsigma_1,\btau}(\mathbf{A}_t(i,j) = 1)P_{\bsigma_2,\btau}(\mathbf{A}_t(i,j) = 1)}{P_{0,n}(\mathbf{A}_t(i,j) = 1)}\right.\\
            &\qquad\qquad\qquad\qquad\qquad\qquad \left.+\frac{P_{\bsigma_1,\btau}(\mathbf{A}_t(i,j) = 0)P_{\bsigma_2,\btau}(\mathbf{A}_t(i,j) = 0)}{P_{0,n}(\mathbf{A}_t(i,j) = 0)}\right).
        \end{aligned}
    \end{equation}
    Since each membership $\bsigma\in\mathcal S_n$ has $n/2$ terms as $0$ and $n/2$ as $1$, we can fix $\bsigma_1$ and consider the sum for all $\bsigma_2$. For $\bsigma_1$ and any $\bsigma_2$, define sets 
\begin{align*}
    \Omega_1 = & \{(i,j,t):1\le i<j\le n,t\in[T],2\nmid \bsigma_1(i)+\bsigma_1(j)+\btau(t),2\nmid \bsigma_2(i)+\bsigma_2(j)+\btau(t)\}\\
    \Omega_2 = &\{(i,j,t):1\le i<j\le n,t\in[T],2\nmid \bsigma_1(i)+\bsigma_1(j)+\btau(t),2\mid \bsigma_2(i)+\bsigma_2(j)+\btau(t)\}\\
    \Omega_3 = &\{(i,j,t):1\le i<j\le n,t\in[T],2\mid \bsigma_1(i)+\bsigma_1(j)+\btau(t),2\nmid \bsigma_2(i)+\bsigma_2(j)+\btau(t)\}\\
\Omega_4 = & \{(i,j,t):1\le i<j\le n,t\in[T],2\mid \bsigma_1(i)+\bsigma_1(j)+\btau(t),2\mid \bsigma_2(i)+\bsigma_2(j)+\btau(t)\}
\end{align*}
    Let $c=|\bsigma_1^{-1}(1)\cap \bsigma_2^{-1}(1)|$. The sizes of $\Omega_j,j\in[4]$ satisfy:
    \begin{equation}
        \begin{aligned}
            |\Omega_1|+|\Omega_4| &=(2c^2+2(n/2-c)^2-n/2)T_n,\\
            |\Omega_2|+|\Omega_3| &= 4T_n(c(n/2-c))\,.
        \end{aligned}
    \end{equation}
The RHS of \eqref{eq21} can be rewritten by combining the same terms according to the membership of the triplet $(i,j,t)$ in $\Omega_k$ ($1\le k\le 4$):
        \begin{align}
            &\sum_{\mathbf{A}}\frac{P_{1,\btau,n}^2(\mathbf{A})}{P_{0,n}(\mathbf{A})}\nonumber\\
            =& \frac{1}{\binom{n}{n/2}^2}\sum_{\bsigma_1}\sum_{\bsigma_2}\left(\frac{3\rho_n/2\cdot 3\rho_n/2}{\rho_n}+\frac{(1-3\rho_n/2)\cdot(1-3\rho_n/2)}{1-\rho_n}\right)^{|\Omega_1|}\nonumber\\
            &\quad \times\left(\frac{3\rho_n/2\cdot \rho_n/2}{\rho_n}+\frac{(1-3\rho_n/2)\cdot(1-\rho_n/2)}{1-\rho_n}\right)^{|\Omega_2|}\nonumber\\
            &\quad \times\left(\frac{\rho_n/2\cdot 3\rho_n/2}{\rho_n}+\frac{(1-\rho_n/2)\cdot(1-3\rho_n/2)}{1-\rho_n}\right)^{|\Omega_3|}\nonumber\\
            &\quad \times \left(\frac{\rho_n/2\cdot \rho_n/2}{\rho_n}+\frac{(1-\rho_n/2)\cdot(1-\rho_n/2)}{1-\rho_n}\right)^{|\Omega_4|}\nonumber\\
            = & \frac{1}{\binom{n}{n/2}^2}\sum_{\bsigma_1}\sum_{\bsigma_2}\left(\frac{1-3\rho_n/4}{1-\rho_n}\right)^{|\Omega_1|+|\Omega_4|}\left(\frac{1-5\rho_n/4}{1-\rho_n}\right)^{|\Omega_2|+|\Omega_3|}\nonumber\\
            = & \frac{1}{\binom{n}{n/2}}\sum_{c = 0}^{n/2}\binom{n/2}{c}^2\left(\frac{1-3\rho_n/4}{1-\rho_n}\right)^{2T_n(c^2+(n/2-c)^2-n/4)}\left(\frac{1-5\rho_n/4}{1-\rho_n}\right)^{4T_n(c(n/2-c))}. \label{eq:simplified_P(A)}
        \end{align}
    Observe that $2T_n(c^2+(n/2-c)^2-n/4)+4T_n(c(n/2-c)) = {n\choose 2}T_n$, and hence
        \begin{align}
            &\sum_{\mathbf{A}}\frac{P_{1,\btau,n}^2(\mathbf{A})}{P_{0,n}(\mathbf{A})}\nonumber\\
            =& \frac{1}{\binom{n}{n/2}}\left(\frac{(1-3\rho_n/4)(1-5\rho_n/4)}{(1-\rho_n)^2}\right)^{{n\choose 2}T_n/2}\nonumber\\
            &\quad\times\sum_{c = 0}^{n/2}\binom{n/2}{c}^2\left(\frac{1-3\rho_n/4}{1-5\rho_n/4}\right)^{2T_n(c^2+(n/2-c)^2-n/4)-{n\choose 2}T_n/2}\nonumber\\
           \le & \frac{1}{\binom{n}{n/2}}\left(\frac{(1-3\rho_n/4)(1-5\rho_n/4)}{(1-\rho_n)^2}\right)^{{n\choose 2}T_n/2}\nonumber\\
            &\quad\times\sum_{c = 0}^{n/2}\binom{n/2}{c}^2\left(\frac{1-3\rho_n/4}{1-5\rho_n/4}\right)^{2T_n(c-n/4)^2}.\label{eq24}
        \end{align}
    Whenever $\rho_n= o((1/(n\sqrt{T_n}))$ and $n,T_n\rightarrow\infty$, we have
    \[\left(\frac{(1-3\rho_n/4)(1-5\rho_n/4)}{(1-\rho_n)^2}\right)^{{n\choose 2}T_n/2} = \left(1-\frac{\rho_n^2/16}{(1-\rho_n)^2}\right)^{{n\choose 2}T_n/2} \rightarrow 1.\]
    Thus, we can drop the second term in the RHS of \eqref{eq24}. 
    
    Next, let $\kappa=\kappa_n$ be such that $\kappa_n\rightarrow\infty$ and $\kappa_n^2 nT_n\rho_n\rightarrow 0$. For example, under the condition $nT_n\rho_n\rightarrow 0$ we can pick $\kappa_n=(nT_n\rho_n)^{-1/3}$. Define $n_1=n/4-\kappa n^{1/2}$, $n_2=n/4+\kappa n^{1/2}$. 
    We decompose \eqref{eq24} into three terms:
    \begin{align*}
    &\sum_{\mathbf{A}}\frac{P_{1,\btau,n}^2(\mathbf{A})}{P_{0,n}(\mathbf{A})} 
    =(1+o(1))\times\\
    &\frac{1}{\binom{n}{n/2}}\bigg[\underbrace{\sum_{0\le c\le n_1}}_{\displaystyle\mathrm{(I)}}+\underbrace{\sum_{n_1< c <n_2}}_{\displaystyle\mathrm{(II)}}+\underbrace{\sum_{n_2\le c \le n/2}}_{\displaystyle\mathrm{(III)}}\bigg]\binom{n/2}{c}^2\left(\frac{1-3\rho_n/4}{1-5\rho_n/4}\right)^{2T_n(c-n/4)^2}.\end{align*}
    We remark that $\displaystyle\mathrm{(I)} = \displaystyle\mathrm{(III)}$ by the symmetry of combination numbers and only have to handle (I) and (II). Next, we are going to prove that
    \[\displaystyle{\mathrm{(I)}} = o(1)~~~~\text{and}~~~~\displaystyle{\mathrm{(II)}} \leq 1+o(1).\]
    \textbf{Bounding (II).} For any $c\in[n_1,n_2]$,
    \begin{align*}
    &\left(\frac{1-3\rho_n/4}{1-5\rho_n/4}\right)^{2T_n(c-n/4)^2}\leq\left(1+\frac{\rho_n/2}{1-5\rho_n/4}\right)^{2\kappa^2nT_n}\\
    \leq & (1+\rho_n)^{2\kappa^2 nT_n}\leq e^{2\kappa^2 nT_n\rho_n}\rightarrow 1,\end{align*}
    where we use the inequality $(1+x/y)^y\leq e^x$ for any $y>1,|x|\leq y$. Plugging in the above inequality, we have
    \[\displaystyle\mathrm{(II)} \leq \frac{1}{\binom{n}{n/2}}\sum_{c = n_1}^{n_2}\binom{n/2}{c}^2e^{2\kappa^2nT_n\rho_n} \leq 1+o(1).\]
    \textbf{Bounding (I).} We decompose (I) into two terms: for $d\in(0,1/4)$,
    \[\displaystyle{\mathrm{(I)}} = \frac{1}{\binom{n}{n/2}}\bigg[\underbrace{\sum_{c = 0}^{n/4-dn}}_{\displaystyle\mathrm{(IV)}}+\underbrace{\sum_{c = n/4-dn}^{n_1}}_{\displaystyle\mathrm{(V)}}\bigg]\binom{n/2}{c}^2\left(\frac{1-3\rho_n/4}{1-5\rho_n/4}\right)^{2T_n(c-n/4)^2}\,.\] 
    
    For (IV), we have
    \begin{equation}\label{eq25}
        \displaystyle{\mathrm{(IV)}}\leq \frac{\sum_{c = 0}^{n/4-dn}\binom{n/2}{c}^2}{\binom{n}{n/2}}\left(\frac{1-3\rho_n/4}{1-5\rho_n/4}\right)^{n^2T_n/8}.
    \end{equation}
    Observe that $\frac{\sum_{c = 0}^{n/4-dn}\binom{n/2}{c}^2}{\binom{n}{n/2}}$ is the probability of a hypergeometric-$(n,n/2,n/2)$ random variable being less than or equal to $n/4-dn$. 
Using Lemma \ref{l3} with $N = n$, $K = m = n/2$, we obtain
    \[\frac{\sum_{c = 0}^{n/4-dn}\binom{n/2}{c}^2}{\binom{n}{n/2}} = P({\rm HG}(n,n/2,n/2)\leq n/4-dn)\leq e^{-4d^2n}.\]
    Plugging this inequality into \eqref{eq25}, we have
    \[\displaystyle{\mathrm{(IV)}}\leq e^{-4d^2n}e^{\rho_n n^2 T_n} = e^{-n(4d^2-n T_n\rho_n)}.\]
    Thus, $\displaystyle{\mathrm{(IV)}}\rightarrow 0$ for any constant $d>0$ and $n,T_n\rightarrow\infty$.

    For (V), we have 
    \begin{equation}\label{eq26}
        \begin{aligned}
            \displaystyle{\mathrm{(V)}} &=  \frac{1}{\binom{n}{n/2}}\sum_{c = n/4-dn}^{n_1}\binom{n/2}{c}^2\left(\frac{1-3\rho_n/4}{1-5\rho_n/4}\right)^{T_n[2(c-n/4)^2]}\\
            &= \frac{1}{\binom{n}{n/2}}\sum_{r = \kappa \sqrt{n}}^{dn}\binom{n/2}{n/4-r}^2\left(\frac{1-3\rho_n/4}{1-5\rho_n/4}\right)^{2r^2T_n}.
        \end{aligned}
    \end{equation}
    First, observe that
        $$\left(\frac{1-3\rho_n/4}{1-5\rho_n/4}\right)^{2r^2T_n} = \left(1+\frac{\rho_n/2}{1-5\rho_n/4}\right)^{2r^2T_n}\le e^{2r^2T_n\rho_n}\,,$$ 
    Using the Stirling formula, we have 
    $$\binom{n}{n/2}^{-1}\lesssim n^{1/2}\cdot 2^{-n}\,,$$
    and
    \begin{equation}\nonumber
        \begin{aligned}
           \binom{n/2}{n/4-r}^2 &= \left(\frac{(n/2)!}{(n/4-r)!(n/4+r)!}\right)^2\lesssim n^{-1}\cdot 2^n\left(\frac{n^2/16}{n^2/16-r^2}\right)^{n/2-2r}\left(\frac{n/4}{n/4+r}\right)^{4r}\\
           &= n^{-1}\cdot 2^n\left(1+\frac{1}{n^2/(16r^2)-1}\right)^{n/2-2r}\left(1-\frac{1}{n/(4r)+1}\right)^{4r}\\
           &\le n^{-1}\cdot 2^n\cdot e^{\frac{n/2-2r}{n^2/(16r^2)-1}}\cdot e^{-\frac{4r}{n/(4r)+1}}\\
           &=n^{-1}\cdot 2^n\cdot e^{-r^2 n^{-1}\left(\frac{2}{1/4+r/n}\right)}\\
           &\le n^{-1}\cdot 2^n\cdot e^{-4r^2 n^{-1}}\,,
        \end{aligned}
    \end{equation}
    where the last inequality follows from $r\le dn <n/4$.
    
    Plugging the above three inequalities into \eqref{eq26}, we obtain the following bound:
    \begin{align*}\displaystyle{\mathrm{(V)}} \lesssim  &n^{-1/2}\sum_{r = \kappa\sqrt{n}}^{dn}e^{-2r^2n^{-1}\left(2-nT_n\rho_n\right)}\leq n^{-1/2}\sum_{r = \kappa\sqrt{n}}^{dn}e^{-2r^2/n}\\
    \le & n^{-1/2}\int_{\kappa\sqrt{n}}^\infty e^{-2r^2/n}dr = \int_{\kappa}^\infty e^{-2s^2}ds\lesssim e^{-2\kappa^2}\,,
    \end{align*}
    provided that $nT_n\rho_n\le 1$ and $\kappa \ge 1$.
    In summary, we prove that $\displaystyle{\mathrm{(I)} = \mathrm{(IV)}+\mathrm{(V)}} \rightarrow 0$ and $\displaystyle{\mathrm{(II)}} \rightarrow 1$. These results together show that
    \[\sum_{\mathbf{A}}\frac{P_{1,\btau,n}^2(\mathbf{A})}{P_{0,n}(\mathbf{A})} = 2\cdot \displaystyle{\mathrm{(I)}}+\displaystyle{\mathrm{(II)}}\rightarrow 1. \qedhere\]
\end{proof}

    \begin{lemma}[Eq. (14) in \cite{skala2013hypergeometric}]\label{l3}
        Let $X\sim\text{Hypergeometric}(N,K,m)$ and $p = K/N$. Then for $0<t<mK/N$ we can derive the following bound:
        \[P(X\leq (p-t)m)\leq e^{-2t^2m}.\]
    \end{lemma}

\begin{proof}[Proof of \Cref{thm:info_upper}]
    To simplify notation, we denote 
    $$A=\bsigma^{-1}(0)\,,\quad B=[n]\backslash A\,,\quad C=\btau^{-1}(0)\,,\quad D=[T_n]\backslash C\,.$$

    For disjoint node set $X$ and layer set $Y$, 
    we use $e(X,X,Y)$ to denote the number of distinct edges between nodes in $X$ and in layers indexed by $Y$:
    $$
e(X,X,Y) = \sum_{i<j,(i,j,t)\in X^2\times Y} A_t(i,j)\,.
    $$
Let $X'$ be another node set disjoint with $X$, we use $e(X,X',Y)$ to denote the number of edges  between $X$ and $X'$ in layers indexed by $Y$:
$$
e(X,X',Y) = \sum_{(i,j,t)\in X\times X'\times Y} A_t(i,j)\,.
$$

With this notation, the MLE estimator in \eqref{eq:MLE} can be equivalently written as
$$
\arg\max\max_{A\subseteq[n]:|A|=n/2,C\subseteq[T_n]:|C|=T_n/2}
$$

Let 
\begin{align*}
A_w=&A\cap \hat\bsigma_{\rm mle}^{-1}(1)\,,\quad
B_w=B\cap \hat\bsigma_{\rm mle}^{-1}(0)\,,\\
C_w=&C\cap \hat\btau_{\rm mle}^{-1}(1)\,,\quad
D_w=D\cap \hat\btau_{\rm mle}^{-1}(0)\,.
\end{align*}
    In other words, $A_w$, $B_w$, $C_w$, and $D_w$ are the set of nodes or layers for which the MLE memberships are incorrect. 

Assume $d_{\rm Ham}(\hat{\bsigma}_{\rm mle},\bsigma)\le d_{\rm Ham}(1-\hat{\bsigma}_{\rm mle},\bsigma)$ (otherwise we will simply consider $1-\hat\bsigma_{\rm mle}$ instead).
Then we have $|A_w|=|B_w|=k\in[0,n/4]$, and $|C_w|=|D_w|=s\in[0,T_n/4]$.

Use the correspondence
\begin{align*}
\hat A=&\hat\bsigma^{-1}(0)=(A\backslash A_w) \cup B_w\,,\quad \hat B=\hat\bsigma^{-1}(1)=(B\backslash B_w) \cup A_w\,,\\
\hat C=&\hat\btau^{-1}(0)=(C\backslash C_w) \cup D_w\,,\quad \hat D=\hat\btau^{-1}(1)=(D\backslash D_w) \cup C_w\,.
\end{align*}

Then through the construction of MLE, we have
\begin{align*}
    e(\hat A,\hat A,\hat C)+e(\hat B,\hat B,\hat C)+e(\hat A,\hat B,\hat D)\ge e(A,A,C)+e(B,B,C)+e(A,B,D)\,,
\end{align*}
which is equivalent to
\begin{align}
&e(A\backslash A_w, A\backslash A_w, C_w)+e(A_w,A_w,C_w)+e(B\backslash B_w,B\backslash B_w,C_w)+e(B_w,B_w,C_w)\nonumber\\
&+e(A\backslash A_w, A_w, C\backslash C_w)+e(B\backslash B_w, B_w, C\backslash C_w)+e(A\backslash A_w, B_w,D\backslash D_w)\nonumber\\
&+e(B\backslash B_w, A_w, D\backslash D_w)+e(A\backslash A_w, B\backslash B_w, D_w)+e(A_w, B_w, D_w)\nonumber\\
\le &e(A\backslash A_w, A\backslash A_w, D_w)+e(A_w,A_w,D_w)+e(B\backslash B_w,B\backslash B_w,D_w)+e(B_w,B_w,D_w)\nonumber\\
&+e(A\backslash A_w, A_w, D\backslash D_w)+e(B\backslash B_w, B_w, D\backslash D_w)+e(A\backslash A_w, B_w,C\backslash C_w)\nonumber\\
&+e(B\backslash B_w, A_w, C\backslash C_w)+e(A\backslash A_w, B\backslash B_w, C_w)+e(A_w, B_w, C_w)\,.\label{eq:MLE_ineq}
\end{align}

All terms are sums of distinct independent Bernoulli random variables.  The Bernoulli parameters are $3\rho_n/2$ on the LHS and $\rho_n/2$ on the RHS. The number of terms on each side is the same, which equals
\begin{align*}
N_{k,s}:=&\left(\frac{n}{2}-k\right)\left(\frac{n}{2}-k-1\right)s+k(k-1)s+4\left(\frac{n}{2}-k\right)k\left(\frac{T_n}{2}-s\right)+\left(\frac{n}{2}-k\right)^2s+k^2s\\
\gtrsim & \left(\frac{n}{2}-k\right)^2s + \left(\frac{n}{2}-k\right)k \left(\frac{T_n}{2}-s\right)\\
\gtrsim & n^2s+nT_nk\,.
\end{align*}
Let $E_{k,s}$ denote the event that $|A_w|=k$, $|C_w|=s$. Then using the Bernstein's inequality for the event in \eqref{eq:MLE_ineq}, union bound over the choice of $A_w$, $B_w$, $C_w$, $D_w$, and the fact that $\log{n\choose m}\le m\log(n/m+1)$, we have, for some universal constant $c>0$
\begin{align*}
\mathbb P(E_{k,s})\le & {n/2\choose k}^2{T_n/2\choose s}^2 \exp\left[-c(n^2s+nT_nk)\rho_n\right]\\
\le &\exp\left[-c(n^2s+nT_nk)\rho_n-2 k\log(n/k+1)-2s\log(T_n/s+1)\right]\,.
\end{align*}
Because the function $\log(1+x)/x$ is monotone decreasing on $[1/4,\infty)$ with maximum value less than $1$, we have
$$
s\log(T_n/s+1)\le (T_n/4)\log(5)\le T_n/2\,.  
$$
Thus if $k\ge \epsilon n$ for some fixed constant $\epsilon>0$, we have
\begin{align*}
\mathbb P(E_{k,s})\le & \exp\left[-c(n^2s+nT_nk)\rho_n-2k\log (1+\epsilon^{-1})-T_n\right]\\
\le &\exp\left[-c n^2\rho_n s-(c/2) n T_n k\rho_n\right]
\end{align*}
whenever
$$
nT_n\rho_n\ge 8c^{-1}\log(1+\epsilon^{-1})\,,\quad\text{and } n^2\rho_n\ge 4/(c\epsilon)\,.
$$
Then we have the error bound, for large enough $n$,
\begin{align*}
    &P_{1,n}(\ell_n(\hat\bsigma_{\rm mle},\bsigma)\ge 2\epsilon)\\
    \le & \sum_{k\in[\epsilon n, n/4], s\in[0,T_n/4]}
    \exp\left[-c n^2\rho_n s-(c/2) n T_n k\rho_n\right]\\
    = &\sum_{k=\epsilon n}^{n/4}e^{-(c/2)n T_n\rho_n k} \sum_{s=0}^{T_n/4}e^{-cn^2\rho_n s}\\
    \le & \frac{2}{c nT_n\rho_n}\times \frac{1}{cn^2\rho_n} = o(1)\,,
\end{align*}
where the last inequality uses $\sum_{k=a}^b e^{-\beta k}\le \beta^{-1}\int_{a\beta}^\infty e^{-u}du\le \beta^{-1}e^{-a\beta}$.
\end{proof}